\documentclass[11pt]{article}

\usepackage{amsfonts}
\usepackage{amssymb}
\usepackage{amsmath}
\usepackage{amsthm}
\usepackage{graphicx}
\usepackage{color}

\newtheorem{notation}{Notations}[section]
\newtheorem{remarque}[notation]{Remark}
\newtheorem{thm}[notation]{Theorem}

\newtheorem{cor}[notation]{Corollary}
\newtheorem{prop}[notation]{Proposition}

\newcommand{\gga}{\gamma}            
\newcommand{\gd}{\delta}

\newcommand{\gl}{\lambda}

\newcommand{\eps}{\varepsilon}

\newcommand{\R}{\mathbb R}

\newcommand{\N}{\mathbb{N}}

\newcommand{\cC}{\mathcal C}
\newcommand{\cH}{\mathcal{H}}
\newcommand{\cM}{\mathcal M}
\newcommand{\cD}{\mathcal D}

\newcommand{\dive}{\operatorname{div}}

\newcommand{\bEm}{\mathbb{E}_m}
\newcommand{\FCb}{\mathcal{F}\cC_b^1}
\newcommand{\vgamma}[1]{\gga(#1)}
\newcommand{\Pgamma}[1]{P_\gga(#1)}
\newcommand{\normH}[1]{|#1|_H}
\newcommand{\totvar}[1]{\int_X \normH{D_\gga #1} }
\newcommand{\totvarn}[1]{\int_{\R^m} |D_\gga #1| }

\newcommand{\Ldeufin}{L^2_{\gga_m} (\R^m)}
\newcommand{\Ldeu}{L^2_\gga(X)}

\newcommand{\sprod}[2]{\langle #1, #2 \rangle}
\newcommand{\hath}{\hat{h}}
\newcommand{\dom}{\operatorname{dom}}

\def\Div{\textup{div}\,}
\def\e{\varepsilon}
\def\wto{\rightharpoonup}

\begin{document}
\title{\Large Representation, relaxation and convexity for variational problems in Wiener spaces}

\author{A. Chambolle \footnote{CMAP, CNRS UMR 7641, Ecole Polytechnique,
        91128 Palaiseau, France, email: chambolle@cmap.polytechnique.fr},
\and
        M. Goldman
        \footnote{CMAP, CNRS UMR 7641, Ecole Polytechnique,
        91128 Palaiseau, France, email: michael.goldman@polytechnique.edu}
        \and M. Novaga
        \footnote{Dipartimento di Matematica, Universit\`a di Padova,
        via Trieste 63, 35121 Padova, Italy, email: novaga@math.unipd.it}
}
\date{}
\maketitle

\begin{abstract}
\noindent We show convexity of solutions to a class of convex variational problems in the Gauss and in the Wiener space. An important tool in the proof is a representation formula for integral functionals in this infinite dimensional setting, that extends analogous results valid in the classical Euclidean framework.
\end{abstract}

\section{Introduction}

\noindent The aim of this paper is to study the convexity of the minimizers of some variational problems in Wiener spaces. In the Euclidean setting convexity is a widely discussed issue \cite{kawohl}. 
Recently, following previous work by Korevaar \cite{kore} and Alvarez, Lasry and Lions \cite{ALL}, Alter, Caselles and Chambolle 
\cite{ACC,CaCha} showed the convexity of solutions to variational problems involving functionals with linear growth and in particular to the 
prescribed curvature problem. Using quite different techniques, Figalli and Maggi \cite{figmag} proved the convexity of small mass minimizers of this problem.

\noindent The main goal of this paper, is to extend these results to the (finite dimensional) Gauss space 
and to the (infinite dimensional) Wiener space. In this setting, very few results are currently available. 
To the best of our knowledge, the only result in this direction is contained in \cite{CaMiNo}, 
where the authors proved the convexity of the solutions of the isoperimetric problem in convex domains.
More explicitly they prove the following:

\begin{thm}\emph{\cite{CaMiNo}}\label{CaMiNointro}
Let $C$ be a convex set of positive (Gaussian) measure and of finite (Gaussian) perimeter, then there exists $\alpha>0$ such that for every $v\in[\alpha, \vgamma{C}]$, the solution of the constrained isoperimetric problem
\[
\min\left\{\Pgamma{E} \; :\; E\subseteq C \ {\rm and\ } \vgamma{E}=v\right\}
\]
has a unique solution which is convex.
\end{thm}
 
\noindent We are interested in the convexity of solutions of the problem

\begin{equation}\label{princprob}
\min_{\vgamma{E}=v} \Pgamma{E}-\int_E g(x) \, d\gga(x),
\end{equation}
where $g$ is a convex function.

\noindent The idea is to follow the approach of Caselles and Chambolle \cite{CaCha} in the Euclidean case. We will thus be naturally led to consider the variational problem
 
 \begin{equation}\label{princprobu}
 \min_{ BV_\gga \cap \Ldeu} \totvar{u} +\frac{1}{2} \int_X (u-g)^2 d\gga
 \end{equation}
for which we will show convexity of the minimizers. More generally, we will prove that minimizers of 
\begin{equation}\label{princprobF}\min_{  \Ldeu} \int_X F(D_\gga u)\, d\gga +\frac{1}{2} \int_X (u-g)^2 d\gga\end{equation}
are convex if $F$ and $g$  are convex (see Theorems \ref{convfin} and \ref{convlinear}).

\noindent Extending the variational methods from Euclidean to Wiener spaces is now a quite active field. In particular extending the theory of functions of bounded variation to this setting started with the work of Fukushima \cite{fuku} and Fukushima and Hino \cite{fukuhino}. Since then the properties of $BV$ functions and sets of finite perimeter have been investigated by Ambrosio and his collaborators, see \cite{AMMP} in particular but also \cite{AMP} and \cite{AF}. We also refer to the paper  \cite{GNmodica} where relaxation of the perimeter, isoperimetry and symmetrization are investigated with application to a kind of Modica-Mortola result. 
We must point out that this theory of $BV$ functions is strongly linked with older works of Ledoux and Malliavin \cite{ledoux}, \cite{malliavin}.  \\

\noindent The plan of the paper is the following. 
In Section \ref{not} we recall some notation about the Wiener space and functions of bounded variation. In Section \ref{repr} we prove a useful representation formula for integral functionals on Wiener spaces. In Section \ref{secfin} we show the convexity of the minima of \eqref{princprobu} in finite dimension, and in Section \ref{secinf} we investigate the convexity of the minimizers in the infinite dimensional Wiener space.\\

\section{Notation and preliminary results}\label{not}

\noindent A clear and comprehensive reference on the Wiener space is the book by Bogachev \cite{boga} (see also \cite{malliavin}). 
We follow here closely the notation of \cite{AMMP}. Let $X$ be a separable Banach space and let $X^*$ be its dual. We say that $X$ is a Wiener space if it is endowed with a non-degenerate centered Gaussian probability measure $\gga$. That amounts to say that $\gga$ is a probability measure for which $x^*_\sharp \gga$ is a centered Gaussian measure on $\R$ for every $x^*\in X^*$. The non-degeneracy hypothesis means that $\gga$ is not concentrated 
on any proper subspace of $X$.

\noindent As a consequence of Fernique's Theorem \cite[Theorem 2.8.5]{boga}, for every $x^*\in X^*$, the function $R^* x^*(x)=\sprod{x^*}{x}$ is in $\Ldeu=L^2(X,\gga)$. Let $\cH$ be the closure of $R^*X^*$ in $\Ldeu$; the space $\cH$ is usually called the reproducing kernel Hilbert space of $\gga$. Let $R$, the operator from $\cH$ to $X$, be the adjoint of $R^*$ that is, for $\hath \in \cH$,
\[R \hath=\int_X x \hath(x)\, d\gga \]
where the integral is to be intended in the Bochner sense. It can be shown that $R$ is a compact and injective operator \cite{boga}. We will let $Q=RR^*$ so that for every $x^*$, $y^* \in X^*$,
\[\langle Q x^* ,y^* \rangle=\int_X \langle x^*,x\rangle \langle y^*, x\rangle \ d\gga.\]
 We denote by $H$ the space $R\cH \subset X$. This space is called the Cameron-Martin space. It is a separable Hilbert space with the scalar product given by
\[
[h_1,h_2]_H=\sprod{\hath_1}{\hath_2}_{\Ldeu}
\]
if $h_i=R\hath_i$. We will denote by $\normH{\cdot}$ the norm in $H$. 
The space $H$ is a dense subspace of $X$, with compact embedding, and $\gga(H)=0$ if $X$ is of infinite dimension. 

\noindent For $x_1^*, .., x_m^*\in X^*$ we denote by $\Pi_{x_1^*, .., x_m^*}$ the projection from $X$ to $\R^m$ given by
\[\Pi_{x_1^*,.., x_m^*}(x)=(\sprod{x_1^*}{x},.., \sprod{x_m^*}{x}).\]
We will also denote it by $\Pi_m$ when  specifying the points $x_i^*$ is unnecessary. Two elements $x_1^*$ and $x_2^*$ of $X^*$ will be called orthonormal if the corresponding $h_i=Qx_i^*$ are orthonormal in $H$ (or equivalently if $x_1^*$ and $x_2^*$ are orthonormal in $\Ldeu$). We will fix in the following an orthonormal basis of $H$ given by $h_i=Q x_i^*$.

\noindent We also denote by $H_m=\textrm{span}(h_1, .., h_m)$ and $X_m^\perp=\textrm{Ker}(\Pi_m)=\overline{H_m^\perp}^X$,
so that $X\cong \R^m\oplus X_m^\perp$.
The map $\Pi_m$ induces the decomposition $\gga=\gga_m \otimes\gga_m^\perp$, with $\gga_m,\,\gga_m^\perp$ Gaussian measures
on $\R^m,\,X_m^\perp$ respectively.

\begin{prop}\emph{\cite{boga}}
Let $\hath_1, .., \hath_m$ be in $\cH$ then the image measure of $\gga$ under the map
\[\Pi_{\hath_1, .., \hath_m} (x)= (\hath_1(x), .. , \hath_m(x))\]
is a Gaussian in $\R^m$. If the $\hath_i$ are orthonormal, then such measure is the standard Gaussian measure on $\R^m$.
\end{prop}

\noindent Given $u\in \Ldeu$, we will consider the canonical cylindrical approximation $\bEm$ given by
\[\bEm u (x)=\int_{X_m^\perp} u(\Pi_m(x),y) \,d\gga_m^\perp(y).\]
Notice that $\bEm u$ depends only on the first $m$ variables 
(we call such function a cylindrical function) and $\bEm u$ converges  to $u$ in $\Ldeu$.

\noindent We will denote by $\FCb(X)$ the space of all cylindrical  $\cC^1$ bounded functions that is the functions of the form $v(\Pi_m (x))$ with $v$ a $\cC^1$ bounded function from $\R^n$ to $\R$. 
We denote by $\FCb(X,H)$ the space generated by all functions of the form 
$\Phi h$, with $\Phi\in \FCb(X)$ and $h\in H$.

\noindent We now give the definitions of gradients, Sobolev spaces and functions of bounded variation. 
Given $u: X\rightarrow \R$ and $h=R\hath \in H$, we define
\[\frac{\partial u} {\partial h} (x)=\lim_{t\to 0} \frac{u(x+th)-u(x)}{t}\]
whenever the limit exists, and
\[\partial_h^* u= \frac{\partial u}{\partial h} - \hath u. \]
We define $\nabla_H u: X\rightarrow H$,  the gradient of $u$ by
\[\nabla_H u= \sum_{i=1}^{+\infty} \frac{\partial u}{\partial h_i}\,  h_i\]
and the divergence of $\Phi: X\rightarrow H$ by
\[\dive_\gga \Phi =\sum_{i=1}^{+\infty} \partial_{h_i}^* [\Phi,h_i]_H.\]

\noindent The operator $\dive_\gga$ is the adjoint of the gradient so that for every $u \in \FCb(X)$ and every $\Phi \in \FCb(X,H)$, the following integration by parts holds:
\begin{equation}\label{integpart}
\int_X u \dive_\gga \Phi \,d\gga =-\int_X [\nabla_H u, \Phi]_H d\gga.
\end{equation}
\noindent The $\nabla_H$ operator is closable in $L^2_\gga(X)$ and we will denote by $H^{1}_\gga(X)$ its closure in $L^2_\gga(X)$. Formula \eqref{integpart} still holds for $u \in H^{1}_\gga(X)$ and $\Phi \in \FCb(X,H)$. Analogously, we define the Sobolev spaves $W^{1,p}_\gga(X)$ for $p\ge 1$ (these spaces are denoted by $\mathbb{D}^{1,p}(X,\gga)$ in \cite{AMMP}).

\noindent Following \cite{fuku,AMMP}, given $u\in L^1_\gga(X)$ we say that $u\in BV_\gga(X)$ if
\[\totvar{u}=\sup \left\{ \int_X u \dive_\gga \Phi \, d\gga; \; \Phi \in \FCb(X,H), \; \normH{\Phi}\le 1 \; \forall x\in X\right\}<+\infty.\] 
We will also denote by $|D_\gga u|(X)$ the total variation of $u$. If $u=\chi_E$ is the characteristic function of a set $E$ we will denote $P_\gga(E)$ its total variation and say that $E$ is of finite perimeter if $\Pgamma{E}$ is finite. 

\noindent Let $\cM(X,H) $ be the set of countably additive measure on $X$ with values in $H$ with finite total variation. As shown in \cite{AMMP} we have the following properties of $BV_\gga(X)$ functions.

\begin{thm}\label{defBV}
Let $u\in BV_\gga(X)$ then the following properties hold:
\begin{itemize}
\item there exists  a measure $D_\gga u\in\cM(X,H)$  such that for every $\Phi \in \FCb(X)$ we have:
\[\int_X u \, \partial^*_{h_j} \Phi \; d \gga=- \int_X \Phi d\mu_j \qquad \forall j\in \N\]
where $\mu_j=[h_j,D_\gga u]_H$.
\item $|D_\gga u|(X)=\inf \varliminf \{ \int_X \normH{\nabla_H u_i} d\gga \;: \; u_j \in W^{1,1}_\gga(X), \; u_j\rightarrow u \textrm{ in } L^1_\gga(X) \}$.
\end{itemize}
\end{thm}
\noindent We next introduce the the Ornstein-Uhlenbeck semigroup. Let $u \in L^1_\gga(X)$ then 
\[T_t u(x) := \int_X u\left(e^{-t} x+\sqrt{1-e^{-2t}}y\right) \, d\gga(y).\]

\begin{prop}\label{ornstein}
 The Ornstein-Uhlenbeck semigroup satisfies:
\begin{itemize}
 \item if $u \in L^p_\gga(X)$ with $p>1$ then $T_t u \in W^{1,1}_\gga(X)$,
\item if $u \in L^p_\gga(X)$  with $p\ge 1$ then $T_t u$ converges in $L^p_\gga(X)$ to $ u$ when $t$ goes to zero,
\item for every $\Phi \in \FCb(X,H)$, and $u\in \Ldeu$,
\begin{equation}\label{integpartOU}\int_X T_t u \dive_\gga \Phi \, d\gga= e^{-t} \int_X u \dive_\gga T_t \Phi\,  d\gga,\end{equation}
\item if $\Phi \in \FCb(X,H)$ then $T_t \Phi \in \FCb(X,H)$,
\item for every convex function $F: H \to \R\cup \{+\infty\}$, and every $\Phi$,
\begin{equation}\label{jensF}\int_X F(T_t \Phi) d\gga \le \int_X F(\Phi) \, d\gga.\end{equation} 
\end{itemize}
\end{prop}
\noindent The proof of this proposition can be found in \cite{AMMP}. The only additional property here is \eqref{jensF} which follows from Jensen's inequality and the rotation invariance of the measure $\gga$. 
\begin{remarque}\rm
 Notice that \eqref{integpartOU} holds more generally for $u$ in the Orlicz space $L\log^{\frac{1}{2}}L$ but not for a general $u$ in $L^1_\gga(X)$ (see \cite{AMMP}).
\end{remarque}
\begin{prop}
Let $u=v(\Pi_m)$ be a cylindrical function then $u\in BV_\gga(X)$ if and only if $v\in BV_{\gga_m}(\R^m)$. We then have
\[\totvar{u}=\totvarn{v}.\]
\end{prop}

\begin{prop}[Coarea formula \cite{AFP}]\label{procoarea}
If $u\in BV_\gga(X)$ then for every Borel set $B\subset X$, 
\begin{equation}\label{coarea}|D_\gga u|(B)=\int_{\R} \Pgamma{\{u>t\},B} \, dt.\end{equation}
\end{prop}

\noindent The following result can be found in \cite[Corollary 4.4.2]{boga}.
\begin{prop}\label{contconv}
Let $u$ be a convex function from $X$ to $\R\cup \{+\infty\}$, let $F(t)=\vgamma{ \{u\le t\}}$ and $t_0=\inf \{ t  : F(t)>0 \}$, 
then $F$ is continuous on $\R \backslash \{t_0\}$. As a consequence $\vgamma{\{u=t\}}=0$ for every $t\neq t_0$.
\end{prop}
 
\smallskip

\noindent In the finite dimensional setting, we will keep the same notations as in the infinite dimensional one. Notice that in $\R^m$, 
the following equality holds:
\[\dive_\gga \Phi= \dive \Phi -\sprod{x}{\Phi}.\]

\noindent We see that functions in $BV_{\gga_m}(\R^m)$ are in $BV_{\textrm{loc}}(\R^m)$ and that $D_{\gga_m} u= \gga_m Du$ so that most of the properties of classical $BV$ functions extend to $BV_{\gga_m}(\R^m)$ (see \cite{AFP}). \\

\noindent For $F:H \to \R$ a convex function we denote by $F^*$ its convex conjugate defined for $\Phi \in H$ by
\[F^*(\Phi) :=\sup_{h \in H} \ [\Phi,h]_H -F(h)\]
and by $F^\infty$ its recession function defined for $h \in H$ as:
\[F^\infty(h):=\lim_{t \to +\infty} \frac{F(th)}{t}.\]
For the main properties of these functions we refer to \cite{livrerockafellar}. The main assumptions we will make are:
\begin{itemize}
 \item[(H1)] $F: H\to \R\cup\{+\infty\}$ is a proper  (i.e. $F$ is not identically plus infinity) convex lower semi-continuous (l.s.c.), bounded from below and attains its minimum.
\item[(H2)] $F$ has $p\ge 1$ growth i.e. there exists $\alpha_1$, $\beta_1$, $\alpha_2$ and $\beta_2$  real positive such that
\[\alpha_1 \normH{h}^p+\beta_1\ge F(h)\ge \alpha_2 \normH{h}^p -\beta_2 \qquad \forall h \in H.\]
\end{itemize}
Notice that a convex function satisfying (H2) also satisfies (H1). We observe that hypothesis (H2) includes the limiting case $p=1$ which is of particular interest for us. Under hypothesis (H1), it is not restrictive to assume that  $F(0)=0$ and $F\ge 0$.\\

\noindent By Hahn-Banach Theorem, for every proper convex l.s.c.~function $F: H\to \R\cup\{+\infty\}$, there exists $q\in H$ such that $F'(h):= F(h)-[q,h]_H$ satisfies (H1).

\section{Representation formula and relaxation of integral functionals}\label{repr}
We extend in this section a representation formula for integral functionals. We start by proving it for functionals with linear growth. 
\begin{prop}\label{thmDemTem}
 Let $F : H \to \R$ be a convex function satisfying
\[\alpha |h|+\beta\ge F(h)\ge 0 \qquad \forall h \in H\]
For $\mu \in \cM(X,H)$, with $\mu=\mu^a \gga +\mu^s$ its Radon-Nikodym decomposition, let
\[\int_X F(\mu) := \int_X F(\mu^a) d\gga+\int_X F^{\infty}\left(\frac{d \mu^s}{d|\mu^s|} \right)d|\mu^s|,\]
then there holds
 \begin{equation}\label{eqmucyl}\int_X F(\mu)= \sup_{\Phi \in \FCb(X,H)} \int_X [\Phi, d\mu]_H -\int_X F^*(\Phi)d\gga,\end{equation}
\end{prop}
where $[\Phi,d\mu]_H:= [\Phi, \frac{d\mu}{d|\mu|}]_H d|\mu|$.
\begin{proof}
 
For $\mu \in \cM(X,H)$, with $\mu=\mu^a \gga +\mu^s$ its Radon-Nikodym decomposition let  $\cD_F:=\{\Phi =\sum_{i=1}^n \chi_{A_i} h_i \,:\, n\in \N, \ A_i \textrm{ disjoint Borel sets, } h_i \in H, \ F^*(h_i) <+\infty\}$. Then we start by proving 
\begin{equation}\label{eqtilde}
\int_X F(\mu)=\sup_{\Phi\in \cD_F}  \int_X [\Phi, d\mu]_H -\int_X F^*(\Phi) \, d\gga .
\end{equation}
 The proof is adapted from \cite{DemTem} and is divided into three steps.

\smallskip

\noindent{\it Step 1.}
Let 
\[
M(\mu):=\sup_{\Phi\in \cD_F} \int_X [\Phi, d\mu]_H -\int_X F^*(\Phi) d\gga.
\]
We will show that for every $h \in L^1_\gga(X,H)$,
\begin{equation}\label{eqM}
M(h\gga)=\int_X F(h)d\gga.
\end{equation}
By definition of convex conjugate, it is readily checked that $M(h\gga)\le \int_X F(h)d\gga$.
We thus turn to the other inequality. By definition of the  integral, 
for every $\gd>0$, there exists $h_i \in H$ and $A_i \subset X$ with $A_i$ disjoints Borel sets and $i\in[1,n]$ such that if we set 
\[\theta= \sum_{i=1}^{n} \chi_{A_i} h_i\]
then $|\theta-h|_{L^1_\gga(X,H)} \le \gd$. As $F$ is of linear growth it is Lipschitz continuous and thus we can assume that also
\[|F(h)-F(\theta)|_{L^1_\gga(X)} \le \delta.\]

\noindent For every $i$, by definition of convex conjugate, there exists $\xi_i \in H$  such that
\[
F(h_i)\le [\xi_i, h_i]_H -F^*(\xi_i) +\delta.\]
 
\noindent Notice that since $F$ is of linear growth, the $\xi_i's$ are uniformly bounded. From this, setting $\Phi=\sum_{i=1}^n \chi_{A_i} \xi_i$ we have
\begin{align*}
\int_X F(h) d\gga &\le \int_X F(\theta) d\gga + \delta \\
										&= \sum_{i=1}^{n} \int_{A_i} F(h_i) d\gga + \delta\\
										&\le \sum_{i=1}^{n} \int_{A_i} [\xi_i, h_i]_H -F^*(\xi_i) d\gga +2\delta\\
					          &=\int_X [\Phi, \theta]_H -F^*(\Phi) d\gga + 2 \delta \\
						&\leq \int_X [\Phi, h]_H -F^*(\Phi) d\gga + C \delta\\
						&\le M(h)+ C\gd. 
										\end{align*}

\noindent Since $\delta$ is arbitrary we have $M(h\gga)=\int_X F(h)d\gga$.
\smallskip

\noindent{\it Step 2.} 
By reproducing the proof with $F^\infty$ instead of $F$, $\frac{d \mu^s}{d|\mu^s|}$ instead of $h$ and $|\mu^s|$ instead of $\gga$ we find, using that $\cD_{F^\infty}=\cD_F$ (since $\dom \ F^*=\dom \ (F^\infty)^*$ by \cite[Thm. 13.3]{livrerockafellar}) and $(F^\infty)^*=0$ in its domain,
\[M_\infty(\mu^s):=\sup_{\Phi \in \cD_F} \int_X [\Phi,d\mu^s]_H=\int_X F^\infty\left(\frac{d \mu^s}{d |\mu^s|}\right) d|\mu^s|.\]
\smallskip

\noindent{\it Step 3.} It remains to show that
\[
M(\mu^a \gga+\mu^s)=M(\mu^a \gga)+ M_\infty(\mu^s).
\] 
One inequality is easily obtained, since
\begin{align*}
M(\mu^a \gga+\mu^s)&=\sup_{\Phi\in \cD_F} \int_X [\Phi, \mu^a]_H d\gga +\int_X [\Phi, d\mu^s]_H-\int_X F^*(\Phi) d\gga\\
												&\le \left(\sup_{\Phi \in \cD_F} \int_X [\Phi, \mu^a]_H -F^*(\Phi)d\gga \right) + \left( \sup_{\Phi \in \cD_F} \int_X [\Phi,d\mu^s]_H\right)\\
												&=M(\mu^a \gga)+ M_\infty(\mu^s).
												\end{align*}
For the opposite inequality, let $\delta>0$ be fixed then there exists $\Phi_1$ and $\Phi_2$ such that
\begin{align*}
M(\mu^a \gga) &\le \int_X [\Phi_1, \mu^a]_H -F^*(\Phi_1)d\gga  + \delta\\
M_\infty(\mu^s)&\le \int_X [\Phi_2, d\mu^s]_H + \delta.
\end{align*}
Taking $\Phi$ equal to $\Phi_2$ on a sufficiently small neighborhood of the support of $\mu^s$ and equal to $\Phi_1$ outside this neighborhood, by the regularity of the measures $\mu^a \gga$ and $\mu^s$ we get
\begin{align*}
M(\mu^a\gga)+M_\infty(\mu^s)&\le\int_X [\Phi, \mu^a]_H -F^*(\Phi)d\gga  + \int_X [\Phi, d\mu^s]_H + C\delta\\
&\le M( \mu^a \gga+ \mu^s)+C\delta
\end{align*}
which gives the opposite inequality and shows \eqref{eqtilde}.\\

\noindent For $\Phi \in \cD_F$, the image of $\Phi$, being a finite number of vectors of $H$, is included in a finite dimensional vector space $V$ of $H$. If we now consider $K$ the convex hull of these vectors then $K$ is a convex polytope of $V$. We can then write $\Phi=\sum_{i=1}^N \theta_i \tilde h_i$ with $\tilde h_i$ the extremal points of $K$ and $\theta_i \in L^1_\gga(X)\cap L^1_\mu(X)$ with $\theta_i \ge 0$ and $\sum_{i=1}^N \theta_i \le 1$. Arguing as in \cite[Section 2.1]{AMMP}, $\gga +|\mu|$ being tight we can approximate $\theta_i$ in  $L^1_\gga(X)\cap L^1_\mu(X)$ with $ \theta^k_i \in \FCb(X)$ in such a way that $\theta_i^k\ge 0$ and $\sum_{i=1}^{N} \theta_i^k \le 1$. As $F^*$ is bounded and continuous on $K$, letting $\Phi^k:= \sum_{i=1}^N \theta_i^k \tilde h_i$ we have $\Phi^k \in \cD_F$ and  
\[\lim_{k \to +\infty} \int_X [\Phi^k,d\mu]-\int_X F^*(\Phi^k) d\gga = \int_X [\Phi,d\mu]-\int_X F^*(\Phi) d\gga.\]
\end{proof}

\noindent We then deduce the following corollary:
\begin{thm}\label{corDemTem}For $F : H \to \R\cup \{+\infty\}$  a proper l.s.c.~convex function and $\mu \in \cM(X,H)$, with $\mu=\mu^a \gga +\mu^s$, then again 
 \[
\int_X F(\mu)= \sup_{\Phi \in \FCb(X,H)} \int_X [\Phi, d\mu]_H -\int_X F^*(\Phi)d\gga.
 \]
\end{thm}
\begin{proof}
\noindent \textit{Case 1.} First assume that (H1) holds. For $n \in\N$ let 
\[F_n(p):=\sup_{\normH{\Phi}\le n} [\Phi,p]_H-F^*(\Phi).\]
Then $F_n$ is of linear growth and $F_n$ is a nondecreasing sequence converging pointwise to $F$ and thus by the monotone convergence theorem,
\[\int_X F(\mu^a) \, d\gga =\lim_{n\to \infty} \int_X F_n (\mu^a)\, d\gga.\]
 Analogously, $(F_n)^\infty$ converges monotonically to $F^\infty$. Indeed, since $F_n$ is nondecreasing, $(F_n)^\infty$ is clearly nondecreasing and
\[(F_n)^\infty(p)=\lim_{t \to +\infty} \frac{F_n(tp)}{t} \le \lim_{t \to +\infty} \frac{F(tp)}{t}=F^\infty(p).\]
On the other hand, for every $\Phi \in \dom \ F^*=\dom \ (F^\infty)^*$, if $n \ge \normH{\Phi}$, for every $p \in H$ and $t>0$,
\[\frac{F_n(tp)}{t} \ge [\Phi,p]_H-\frac{F^*(\Phi)}{t}\]
and thus letting $t$ goes to infinity and then $n$ goes to infinity as well, we find
\[\lim_{n \to \infty} (F_n)^\infty(p) \ge \sup_{\Phi \in \dom F^*} [\Phi,p]_H=F^\infty(p).\]
We thus have
\[\int_X F\left(\frac{d \mu^s}{d|\mu^s|} \right)\, d|\mu^s| =\lim_{n\to \infty} \int_X F_n \left(\frac{d \mu^s}{d|\mu^s|} \right)\, d|\mu^s|.\]

 \noindent By Proposition \ref{thmDemTem}, for every $n \in \N$,
\begin{equation}\label{supn} \int_X F_n (\mu^a)d\gga+\int_X F_n \left(\frac{d \mu^s}{d|\mu^s|} \right)d|\mu^s|= \sup_{\stackrel{\Phi \in \FCb(X,H)}{|\Phi|_\infty \le n}} \int_X [\Phi,d\mu] -\int_X F^*(\Phi) \, d\gga.\end{equation}
Passing to the limit when $n$ tends to infinity we get
\[\int_X F(\mu) =\sup_{\Phi \in \FCb(X,H)} \int_X [\Phi,d\mu] -\int_X F^*(\Phi) \, d\gga.\]

\noindent\textit{Case 2.} Let now $F$ be a generic proper l.s.c.~convex function and  $q\in H$ be such that $F'(h):=F(h)-[q,h]$ satisfies (H1). It is readily seen that $(F')^\infty(h)=F^\infty(h) -[q,h]_H$ and $(F')^*(\Phi)=F^*(\Phi+q)$. Since \eqref{eqtilde} holds for $F'$,
\begin{align*}
 \int_X F(\mu)-\int_X [q,d\mu]_H&=\int_X F'(\mu) \\
&=\sup_{\Phi \in \FCb(X,H)} \int_X [\Phi, d\mu]_H -\int_X F^*(\Phi+q)d\gga\\
	&=\sup_{\Phi \in \FCb(X,H)} \int_X [\Phi-q, d\mu]_H -\int_X F^*(\Phi)d\gga\\
	&=\sup_{\Phi \in \FCb(X,H)} \left\{\int_X [\Phi, d\mu]_H -\int_X F^*(\Phi)d\gga\right\}-\int_X [q,d\mu]_H.
\end{align*}

\end{proof}

\begin{remarque}\rm
 An important example of functionals covered by the Theorem is given by the functionals with $p\ge1$ growth.
\end{remarque}

\noindent For $F$ a proper l.s.c.~convex function, we can define the functional on $\Ldeu$
\begin{equation}\label{repsupF}
 \int_X F(D_\gga u):=\sup_{\Phi \in \FCb(X,H)} \int_X -u \dive_\gga \Phi - F^*(\Phi) \ d\gga.
\end{equation}
\noindent The functional defined in this way is thus l.s.c.~in $\Ldeu$.  By  \eqref{eqmucyl}, we have

\begin{equation}\label{equivF}\int_X F(D_\gga u)= \int_X F(\nabla u) d\gga+ \int_X F^\infty\left(\frac{d D^s_\gga u}{d|D^s_\gga u|}\right) d|D^s_\gga u| \end{equation}
for $u \in BV_\gga(X)$ with $D_\gga u =\nabla u \gga +D^s_\gga u$ its Radon-Nikodym decomposition. \\

 \noindent For $Y$ a metric space and $F:Y\to \R$, we define the relaxed functional $\bar F$ of $F$ by
\[\bar F(x):= \inf_{x_n\to x} \varliminf_{n\to\infty} F(x_n)\,.
\] 
\noindent We then have the following relaxation result:
\begin{prop}\label{relax}
 Let $F$ be a proper l.s.c.~convex function then the functional $\int_X F(D_\gga u)$ is the relaxation of the functional defined as $\int_X F(\nabla_H u) d\gga$ for $u \in W_\gga^{1,1}(X)$. If $F$ satisfies also (H2) then is is also the relaxation of the functional $\int_X F(\nabla_H u) d\gga$ defined on the smaller class $ \FCb(X)$.
\end{prop}

\begin{proof}
 \textit{Case 1.} Assume first that $F$ satisfies (H1). We start by proving that
\begin{equation}\label{relaxW11}
\int_X F(D_\gga u) = \inf \varliminf \left\{\int_X F(\nabla_H u_n) \, d\gga \,:\,
u_n \in W^{1,1}_\gga(X)\,, \quad u_n \to u \textrm{ in } \Ldeu \right\}.\end{equation}
 Thanks to Proposition \ref{thmDemTem}, the inequality `$\le$' is obvious. To prove the opposite inequality, we proceed as in \cite[Th.~4.1]{AMMP} by using the Ornstein-Uhlenbeck semigroup. For $u \in \Ldeu$ and $t>0$, thanks to Proposition \ref{ornstein},
\begin{align*}
 \int_X F(D_\gga T_t u) &= \sup_{\Phi \in \FCb(X,H)}\int_X -T_t u \dive_\gga \Phi -F^*(\Phi)\, d\gga\\
			&=\sup_{\Phi \in \FCb(X,H)} \int_X-e^{-t} u \dive_\gga T_t\Phi -F^*(\Phi)\, d\gga\\
			&\le\sup_{\Phi \in \FCb(X,H)} \int_X-e^{-t} u \dive_\gga T_t\Phi -F^*(T_t\Phi)\, d\gga\\
			&\le e^{-t}\sup_{\Phi \in \FCb(X,H)}  \int_X-e^{-t} u \dive_\gga T_t\Phi -F^*(T_t\Phi)\, d\gga\\
			&\le e^{-t} \int_X F(D_\gga u)
\end{align*}
where, as $F(0)=0$ we have $F^* \ge 0$ and thus $e^{-t} F^*\le F^*$ . This inequality shows that  
\[
\int_X F(D_\gga u) \ge 
\inf \varliminf \left\{\int_X F(\nabla_H u_n) \, d\gga\,: \, u_n \in W^{1,1}_\gga(X)\,, \quad u_n \to u \textrm{ in } \Ldeu \right\}.\]

\noindent \textit{Case 2.} Let $F$ be a proper l.s.c.~convex function and $q\in H$ be  such that $F'(h)=F(h)-[q,h]$ satisfies (H1) then for $u\in \Ldeu$,
\[\int_X F(D_\gga u)=\int_X F'(D_\gga u) -\int_X u \dive_\gga p \ d\gga.\]
Therefore, by Case 1 applied to $F'$ we get that 
\[
\int_X F(D_\gga u) = 
\inf \varliminf \left\{\int_X F(\nabla_H u_n) \, d\gga\,: \, u_n \in W^{1,1}_\gga(X)\,, \quad u_n \to u \textrm{ in } \Ldeu \right\}.
\]

\noindent \textit{Case 3.} If now $F$ satisfies (H2), by the density of $\FCb(X)$ in $W^{1,p}_\gga(X)$ for $p\ge 1$, for every $u \in W^{1,p}_\gga(X)$ there exists $u_n\in \FCb(X)$ tending to $u$ in  $W^{1,p}_\gga(X)$ and almost everywhere. Then as $F(\nabla_H u_n) \le \alpha_2 \normH{\nabla_H u_n}^p +\beta_2$, by the dominated convergence theorem,
\[\int_X F(\nabla_H u_n) \ d\gga \to \int_X F(\nabla_H u) \ d\gga. \]
Thus starting from $W_\gga^{1,p}(X)$ or $\FCb(X)$ gives the same relaxation for $\int_X F(\nabla_H u) d\gga$.
\end{proof}

\section{The finite dimensional case}\label{secfin}

In this section we focus on the finite dimensional problem. Let $F:\R^m \to \R$ be a  convex function satisfying for $p\ge 1$,
\[(H_2') \qquad \alpha_2 |h|^p+\beta_2\ge F(h)\ge \alpha |h|^p -\beta \qquad \forall h \in \R^m.\] 
As before we set
\[\int_{\R^m} F(D_{\gga_m} u) \, d\gga_m := 
\sup_{\Phi \in \cC^1_b(\R^m)} \int_{\R^m} \left(-u \dive_\gga \Phi -F^*(\Phi)\right) \, d\gga_m.\]
By Theorem \ref{corDemTem} and Proposition \ref{relax}, 
\[\int_X F(D_{\gga_m} u) = \int_{\R^m} F(\nabla u) d\gga_m+ \int_{\R^m} F^\infty\left(\frac{d D^s_{\gga_m} u}{d|D^s_{\gga_m} u|}\right) d|D^s_{\gga_m} u|\]

\noindent and this functional also coincides with the relaxation for the $\Ldeufin$ topology of the functional classically defined on Lipschitz functions $u$ by $\int_{\R^m} F(\nabla u) d\gga_m$. In this finite dimensional setting this representation formula is not new (see \cite{BDalMaso} and \cite{But}).\\
 We show in this section the convexity of the solutions of
\begin{equation}\label{VP}
\min_{u\in \Ldeufin} \int_{\R^m} F(D_{\gga_m} u) + \frac{(u-g)}{2}^2\,d\gamma_m.
\end{equation}

\noindent Formally the Euler-Lagrange equation of this problem reads
\begin{equation}\label{EL}
-\Div \nabla F(\nabla u) \ +\ x\cdot \nabla F(\nabla u) \ +\ u\ =\ g.
\end{equation}

\begin{thm}\label{convfin}
Let $F:\R^m \to \R$ be a  convex function satisfying (H2') and $g\in \Ldeufin$
be a convex function. The minimizer of \eqref{VP} is then convex.
\end{thm}

\begin{proof} 
\noindent Before entering into the details, let us give a sketch of the proof. We first approximate the functions $F$ and $g$ by smooth quadratic functions $F_n$ and $g_\e$ for which we can use the results of \cite{ALL}. We then construct for this approximating functions, a convex subsolution $u_\eps$ of the problem and consider $u_\eps^n$ the least convex supersolution  of the problem which is greater than $u_\e$. We then show that $u_\e^n$ is in fact a solution of the approximated problem and then let $\e\to0$ and $n\to +\infty$.

\noindent  Let $F_n\to F$ be a sequence of smooth, uniformly convex
functions, with quadratic growth which converge locally uniformly
to $F$. The functional $\int_{\R^m} F_n(\nabla u) d\gga_m$ is then finite if and only if $u \in H^1_{\gga_m}(\R^m)$. Without loss of generality we can assume that $\inf F_n=F_n(0)=\inf F^*_n=F_n^*(0)=0$ and thus $\nabla F_n(0)=\nabla F^*_n(0)=0$.

\noindent  We consider for $\e>0$ the approximation
\[
g_\e(x) \ =\ \max\{g(x),-\frac{1}{\e}\} + \e x^2 + \frac{1}{\e} F_n^*(\e x)
\]
so that $g_\e\to g$ locally uniformly as $\e\to 0$. 
Indeed, it follows from the uniform convexity of $F_n$ that $F_n^*$ is differentiable, hence
\[
\lim_{\e\to 0}\frac{1}{\e} F_n^*(\e x)\ =\ \nabla F^*_n(0)\cdot x\ =\ 0.
\]
Since $F_n(p)\ge  C(|p|^2-1)$, $F^*_n(q) \le C(|q|^2+1)$
and $g_\e\in \Ldeufin$.

\noindent In particular, letting
\[
u_\e(x) \ =\ \frac{F_n^*(\e x)}{\e} + m\e - \frac{1}{\e}\ \in \ \Ldeufin\,,
\]
we have 
\[
-\Div \nabla F_n(\nabla u_\e) \ +\ x\cdot \nabla F_n(\nabla u_\e) \ +\ u_\e
\ =\ -m\e\ +\ \e x^2\ +\ \frac{F_n^*(\e x)}{\e}\ +\ m\e\ -\ \frac{1}{\e}\ 
\le \ g_\e(x)
\]
hence $u_\e$ is a classical subsolution of the approximate problem.
We observe that both $g_\e$ and $u_\e$ have superlinear growth
at infinity.

\noindent We now consider the solution $\bar u $ of
\begin{equation}\label{VPvincol}
\min_{u\ge u_\e}  \int_{\R^m} F_n(\nabla u) + \frac{(u-g_\e)}{2}^2\,d\gamma_m
\end{equation}
which by definition is above $u_\e$. 

\noindent We first show that it is a viscosity supersolution of
\begin{equation}\label{ELapprox}
-\Div \nabla F_n(\nabla u) \ +\ x\cdot \nabla F_n(\nabla u) \ +\ u\ =\ g_\e.
\end{equation}

Let us first notice that by \cite{MicZiem}, the function $\bar u$ is H\"older continuous. Assume that $\bar u$ is not a supersolution of \eqref{VPvincol} then there exists $x_0 \in \R^m$ and a smooth function $\phi$ such that $\phi < \bar u$ in $\R^m \backslash \{x_0\}$, $\phi(x_0)=\bar u(x_0)$ and
\begin{equation}\label{ineqphi}
-\Div \nabla F_n(\nabla \phi) \ +\ x\cdot \nabla F_n(\nabla \phi) \ +\ \phi\ -\ 
g_\e\ <\ 0 \quad \textrm{ at } x_0.
\end{equation}
Replacing $\phi$ by $\phi-\eta |x-x_0|^2$ and noticing that
$\{\phi-\eta |x-x_0|^2 +\delta >\bar u\}\subset B(x_0, \sqrt{\delta/\eta})$
we can further assume by the smoothness of $\phi$, $F_n$ that \eqref{ineqphi} holds on
the open set $\{\phi+\delta >\bar u \}$ for $\delta>0$ small enough.
As $v=\max(\phi+\delta, \bar u)\ge u_\e$, we have
\[\int_{\R^m} F_n(\nabla v)+ \frac{(v-g_\e)^2}{2} d \gamma_m \ge \int_{\R^m} F_n(\nabla \bar u)+ \frac{(\bar u -g_\e)^2}{2} d \gamma_m\]
and thus
\[\int_{\{\phi+\delta>\bar u\}} F_n(\nabla \phi)+\frac{(\phi+\delta-g_\e)^2}{2} d \gamma_m \ge \int_{\{\phi+\delta>\bar u\}} F_n(\nabla \bar u)+\frac{(\bar u-g_\e)^2}{2} d \gamma_m.\]
Using that $F_n(\nabla \bar u)-F_n(\nabla \phi)\ge \nabla F_n(\nabla \phi) \cdot (\nabla \bar u -\nabla \phi)$ by convexity of $F_n$ and 
\[\frac{(\bar u-g_\e)^2}{2}-\frac{(\phi+\delta-g_\e)^2}{2}= \frac{(\phi+\delta-\bar u)^2}{2}+(\phi+\delta-g_\e)(\bar u -\phi-\delta)\]
we get
\begin{align*}
 0&\ge \int_{\{\phi+\delta>\bar u\}}\!\!\! \nabla F_n(\nabla \phi) \cdot (\nabla \bar u -\nabla \phi)+\frac{(\phi+\delta-\bar u)^2}{2}+(\phi+\delta-g_\e)(\bar u -\phi-\delta)d \gamma_m\\
&=\int_{\{\phi+\delta>\bar u\}} \!\!\!\left[-\Div \nabla F_n(\nabla \phi) +x\cdot \nabla F_n(\nabla \phi)+\phi+\delta-g_\e\right](\bar u -\phi -\delta) +\frac{(\phi+\delta-\bar u)^2}{2}d \gamma_m\\
&> 0
\end{align*}
and thus a contradiction. The integration by part used above is justified by the fact that $\{\phi+\delta >\bar u \}$ is an open set on the boundary of which $\phi+\delta$ and $\bar u$ agree.

\noindent Notice that using the same arguments it can be shown that there is no contact between $u_\e$ and $\bar u$ so that $\bar u$ is in fact an unconstrained minimizer of the energy.

\noindent Now, thanks to \cite[Proposition~3]{ALL}, given any supersolution
$u$  of \eqref{ELapprox}, with superlinear growth, the convex
envelope $u^{**}$ is still a supersolution. Moreover, if $u\ge u_\e$,
then clearly $u^{**}\ge u_\e$ (which is convex).

\noindent Hence, if we define $\tilde{u}\le \bar u$ as the infimum of all
supersolutions of \eqref{ELapprox} which are larger than $u_\e$,
it is also the infimum of their convex envelopes (hence it is
a locally uniform limit of convex supersolutions) and therefore
is convex. It is also a supersolution.

\noindent Let us now show that $\tilde{u}$ is a viscosity solution. If it
were not, there would exist a smooth $\phi$ and $x\in \R^m$ with
$\tilde{u}(x)=\phi(x)$, and $\tilde{u}<\phi$ in $\R^m\setminus\{x\}$,
with 
\[
-\Div \nabla F_n(\nabla \phi(x)) \ +\ x\cdot \nabla F_n(\nabla \phi(x))\ +\ \phi(x)\ >\ g_\e(x).
\]
In particular, $\tilde{u}(x)>u_\e(x)$, otherwise $x$ would also
be a local maximum of $u_\e-\phi$ and the reverse inequality should hold.
Now, by standard arguments, we check that $\min\{\tilde{u},\phi-\delta\}$
is still a supersolution, larger than $u_\e$, if $\delta>0$ is small
enough, a contradiction.

\noindent Hence $\tilde{u}$ is a solution of~\eqref{ELapprox}. By \cite[Theorem~4]{Imbert}, $\tilde{u}$ is a $C^{1,1}$ function and thus by \cite[Lemma~2]{ALL}, $\tilde{u}$ satisfies \eqref{ELapprox} almost everywhere (and also weakly). The function $\tilde{u}$ is therefore  a critical point of the (strictly
convex) energy, hence the unique solution to \eqref{VP} (with $F$ replaced
with $F_n$ and $g$ with $g_\e$). Denote now this solution by $u^n_\e$.
\smallskip

\noindent Let us now show that we can send $\e\to 0$ and then $n\to\infty$.

\noindent Comparing the energy of $u^n_\e$ with the energy of $0$, we find that
\begin{equation}\label{inequL}
\|u^n_\e\|_{\Ldeufin} \ \le\ 2\|g_\e\|_{\Ldeufin}
\ \le\ 2\|g\|_{\Ldeufin}+2\left\|\e x^2 + \frac{1}{\e} F_n^*(\e x)\right\|_{\Ldeufin}
\end{equation}
so that $\|u^n_\e\|_{\Ldeufin}$ is uniformly bounded. Hence,
we can send $\e\to 0$ and will find that $u^n_\e\wto u^n$. By a Theorem of Dudley \cite{dudley}, $u^n_\e$ converges locally uniformly to $u^n$ which is thus convex. By the lower-semicontinuity of the energy, $u^n$ is the solution of problem~\eqref{VP} with $F$ replaced with $F_n$.\\

\noindent Analogously,  $u^n\to u$ locally uniformly since by \eqref{inequL}, $\|u^n\|_{\Ldeufin}  \le\ 2\|g\|_{\Ldeufin}$ and thus $u$ is convex. Let us show that $u$ is the minimizer of \eqref{VP}. We start by proving that 
\begin{equation}\label{liminf}
 \varliminf_{n\to \infty} \int_{\R^m} F_n(\nabla u^n) d\gga_m \ge \int_{\R^m} F(\nabla u) d\gga_m.
\end{equation}

\noindent Since $u^n$ is a sequence of convex functions converging to $u \in \Ldeufin$ then, up to subsequence, $\nabla u^n$ converges to $\nabla u$ almost everywhere. Moreover, for all $R>0$ there exists $C=C(R,v)$ such that $\|\nabla u^n\|_{L^\infty(B_R)}\le C$ for all $n\in\mathbb N$. 
This is a general property of convex functions and we refer to \cite[Theorem~3]{CaCha} for further details.
\noindent By Fatou's Lemma, we then get 
\[\varliminf_{n\to \infty} \int_{\R^m} F_n(\nabla u^n) d\gga_m \ge \varliminf_{n\to \infty} \int_{B_R} F_n(\nabla u^n) d\gga_m
=\int_{B_R} F(\nabla u) d\gga_m.\]
Letting $R\to +\infty$ we obtain \eqref{liminf}.

\noindent Now if $v$ is a Lipschitz function in $\Ldeufin$, as $F_n$ converges locally uniformly to $F$,
\[\lim_{n\to \infty} \int_{\R^m} F_n(\nabla v) d\gga_m=\int_{\R^m} F(\nabla v) d\gga_m\]
and thus, by the minimality of $u^n$ and \eqref{liminf},
\begin{align*}
 \int_{\R^m} F(\nabla v) +\frac{(v-g)}{2}^2\,d\gamma_m&=\lim_{n\to \infty}\int_{\R^m} F_n(\nabla v) +\frac{(v-g)}{2}^2\,d\gamma_m\\
							&\ge \varliminf_{n \to \infty} \int_{\R^m} F_n(\nabla u^n) +\frac{(u^n-g)}{2}^2\,d\gamma_m\\
							&\ge \int_{\R^m} F(\nabla u) +\frac{(u-g)}{2}^2\,d\gamma_m.
\end{align*}
Since Lipschitz functions are dense in energy in $\Ldeufin$, we obtain that $u$ is a minimizer of \eqref{VP}.
 \end{proof}

\begin{remarque} \rm
The proof directly extends to variational problems of the form
\[\min_{u\in L^2(\mu)} \int_{\R^m} F(\nabla u) + \frac{(u-g)}{2}^2\,d\mu\]
for measures $d \mu= \mu(x)\,dx$, with $\mu(x)=e^{-(Ax,x)}$ and $A>0$.
\end{remarque}

\begin{remarque} \rm
An other possible approach for proving convexity of the minimizers of \eqref{VP} is to adapt the ideas of Korevaar \cite{kore}, see \cite{theseG} for details. 
\end{remarque}

\begin{remarque}\rm
 Arguing as in the Theorem \ref{convlinear} of the next section, we see that this result extends to generic proper l.s.c.~convex functions $F$.
\end{remarque}

\section{The infinite dimensional case}\label{secinf}
In this final section we return to the infinite dimensional problem. 
\begin{thm}\label{convlinear}
 Let $F: H \to \R\cup \{+\infty\}$ be a proper l.s.c.~convex function  and $g \in \Ldeu$ be a convex function. Then the minimizer of 
\[J(u):=\int_XF(D_\gga u) +\frac{1}{2}\int_X (u-g)^2 d\gga\]
is convex.
\end{thm}

\begin{proof}
\noindent \textit{Case 1.} We start by assuming that $F$ satisfies also (H2) .\\

\noindent Let $g_m=\bEm(g)$ then $g_m$ is a convex function. Let also $\bar u_m$ be the minimizer of
\[
\min_{u\in\Ldeu:\,u=\bEm u} J_m(u):=\int_X F(D_\gga u)+\frac{1}{2}\int_X(u-g_m)^2 d\gga. \]
Thanks to \eqref{equivF}, if $u$ depends only on the first $m$ variables then 
\[\int_X F(D_\gga u) =\int_X F_m(D_\gga u)=\int_{\R^m} F_m(D_{\gga_m} u)\]
where $F_m(h)=F(\Pi_m h)$. By Theorem \ref{convfin}, $\bar u_m$ is thus a convex function. As $J_m(\bar u_m)\le J_m(0)$ and since $g_m\to g$ in $\Ldeu$, 
$\bar u_m$ is bounded in $\Ldeu$ and is thus weakly converging to $\bar u$ which is therefore convex by \cite[Theorem 4.4]{feyel}. \\

\noindent We now show that  $\bar u$ is the minimizer of $J$.\\

\noindent If $u_m$ is a  weakly converging sequence to $u\in \Ldeu$, then by strong convergence of $g_m$ to $g$ we have 
\[\varliminf_{m\to \infty} \frac{1}{2}\int_X|u_m-g_m|^2 d\gga\ge \frac{1}{2}\int_X(u-g)^2 d\gga.\]
By the lower semicontinuity of $\int_X F(D_\gga u)$ (which comes from \eqref{repsupF})  we then have
\[\varliminf_{m \to \infty} J_m(u_m) \ge J(u).\]

\noindent Thus if $u \in \FCb(X)$, by minimality of $\bar u_m$,
\begin{equation}\label{ineqbarucyl}
 J(u)= \lim_{m \to +\infty} J_m(u) \ge \varliminf_{m \to +\infty} J_m (\bar u_m) \ge J(\bar u).
\end{equation}

\noindent Since we assumed that $F$ satisfies (H2), by Proposition \ref{relax}, the space $\FCb(X)$ is dense in energy in $\Ldeu$ and thus inequality \eqref{ineqbarucyl} proves that $\bar u$ is the minimizer of $J$ in $\Ldeu$.\\

  \noindent \textit{Case 2.} If $F$ is a proper l.s.c.~convex function, we can approximate it by a convex function $F_\delta$ with linear growth 
\[F_\delta(p):=\delta \normH{p}+\inf_{q\in H} \left( \frac{1}{\delta}\normH{p-q} +F(q)\right).\]
By Case 1, the minimizer $u_\delta$ of the functional with $F_\delta$ instead of $F$ is convex. As before, we have that $u_\delta$ weakly converges to a convex function $u$ in $\Ldeu$.  As $W_\gga^{1,1}(X)$ is dense in energy in $\Ldeu$, in order to conclude, it is  sufficient to prove that for every $v\in W^{1,1}_\gga(X)\cap \Ldeu$, 
\begin{equation}\label{limsupdelta}\int_X F(\nabla_H v) d\gga \ge \varlimsup_{\delta \to 0} \int_X F_\delta(\nabla_H v) \end{equation}
and 
\begin{equation}\label{liminfdelta}\varliminf_{\delta\to 0} \int_X F_\delta (D_\gga u_\delta) \ge \int_X F(D_\gga u).\end{equation}
For inequality \eqref{limsupdelta} we can assume that $\int_X F(\nabla_H v) d\gga<+\infty$ then as for the Moreau regularization, $\lim_{\delta \to 0} F_\delta(p)= F(p)$ for every $p \in H$ so that for every $v\in W^{1,1}_\gga(X)$, $F_\delta (\nabla_H v)$ converges almost everywhere to $F(\nabla_H v)$ and since $F_\delta(\nabla_H v) \le \delta \normH{\nabla_H v}+F(\nabla_H v)$, by the dominated convergence Theorem, inequality \eqref{limsupdelta} follows.\\

\noindent For inequality \eqref{liminfdelta}, we start by noticing that by calculus on inf-convolutions and convex conjugates, we have,
\[F_\delta^*(q)=\inf_{\stackrel{\normH{p}\le \frac{1}{\delta}}{\normH{p-q}\le \delta}} F^*(p),\]
where we take as a convention that $ F_\delta^*(q)=+\infty$ if $B_{\frac{1}{\delta}}\cap B_\delta(q)=\emptyset$. Therefore, for every $q \in H$, as soon as $\normH{q} \le \frac{1}{\delta}$, we have  $F_\delta^*(q) \le F^*(q)$ and thus
\[\varlimsup_{\delta \to 0} F_\delta^*(q) \le F^*(q) \qquad \forall q \in H.\]
If now $\Phi \in \FCb(X,H)$ with $F^*(\Phi)$ integrable, we have $F_\delta^*(\Phi) \le F^*(\Phi)$ for $\delta$ small enough and thus by the reverse Fatou lemma,
\begin{equation}\label{limsupstar}
 \varlimsup_{\delta \to 0} \int_X F_\delta^*(\Phi) \, d\gga \le \int_X \varlimsup_{\delta \to 0} F^*_\delta (\Phi) d\gga \le \int_X F^*(\Phi) \, d\gga.
\end{equation}
We can now conclude since for every $\Phi \in \FCb(X,H)$ with $\int_X F^*(\Phi)\, d\gga<+\infty$ we have using \eqref{limsupstar},
\begin{align*}
 \varliminf_{\delta \to 0} \int_X F_\delta (D_\gga u_\delta) &\ge \varliminf_{\delta\to 0} \int_X -u_\delta \dive_\gga \Phi -F_\delta^*(\Phi) \,d\gga\\
								&\ge \int_X -u \dive_\gga \Phi -F^*(\Phi) \, d\gga
\end{align*}
Taking then the supremum on all $\Phi \in \FCb(X,H)$ and using \eqref{eqmucyl}, we get \eqref{liminfdelta}.

\end{proof}

\begin{remarque}\rm
Notice that, by taking $F(h)=|h|^p$ with $p\ge 1$, Theorem \ref{convlinear} applies in particular to the 
$p$-Dirichlet problems
\[
\min_{  \Ldeu} \int_X \normH{\nabla_H u}^p\, d\gga +\frac{1}{2} \int_X (u-g)^2 d\gga.
\]
\end{remarque}

\begin{remarque} \rm
 When $X$ is a Hilbert space, there is another definition of the gradient due to Da Prato which gives an alternative definition of Sobolev and BV spaces (see \cite[Section 5]{AMMP}). Roughly speaking it corresponds to $Du:= Q^{-\frac{1}{2}} \nabla_H u$. Theorem \ref{convlinear} then applies to the associated total variation since it is given by the choice
\[F(h)=\left( \sum_{i=1}^{+\infty} \frac{1}{\lambda_i} |h_i|^2\right)^{\frac{1}{2}}\]
where the $\lambda_i$'s are the eigenvalues of $Q$.
\end{remarque}
\begin{remarque} \rm
The proofs of Theorem \ref{convfin} and \ref{convlinear} follow standard $\Gamma$-convergence arguments (see \cite{braides}).
\end{remarque}
\noindent Using the theory of maximal monotone operators \cite{brezis}, we easily get the following corollary.
\begin{cor}
Let $F: H \to \R\cup \{+\infty\}$ be a proper l.s.c.~convex function and let $u_0\in \Ldeu$ be a convex function.
Then the solution $u(t)$ of the $\Ldeu$ gradient flow of  $\int_X F(D_\gga u)$ with initial condition $u(0)=u_0$ is convex for every $t>0$.
\end{cor}

\noindent We can now use these convexity results to show the convexity of solutions of \eqref{princprob}.

\begin{thm}\label{main}
Let  $g$ be a convex function in $\Ldeu$ and let $u$ be the minimizer of \eqref{princprobu}. Let $\overline{\lambda}= \inf\{ \lambda  \; :\;  \gga( u\le \lambda)>0 \}$. If $\overline{v}=\gamma(\{u\le \overline{\lambda}\})$ then for every $v>\overline{v}$ there exists a unique solution to \eqref{princprob} and this solution is convex.
\end{thm}

\begin{proof}
The proof follows quite standard arguments so that we only sketch it (see \cite{CaMiNo} and \cite{ACC} for details).
Let us first consider the problem
\[ 
\min_E \Pgamma{E} + \int_E (g-\lambda)\, d\gga. 
\eqno (P_\lambda) 
\]
Then as in Proposition 34 of \cite{CaMiNo}, by the direct method of the calculus of variations (in $BV_\gga(X;[0,1])$)
 and by the coarea formula it is not difficult to show that $(P_\lambda)$ has a minimum $E_\lambda$. By \cite[Lemma 8]{CaMiNo} we have $E_{\lambda_1} \subset E_{\lambda_2}$ if $\gl_1< \gl_2$. 

\noindent Setting $w(x)=\inf \left\{ \lambda \; :\;  x \in E_\lambda \right\}$, it is not hard to see that $w \in BV_\gga \cap \Ldeu$  and that $w$ solves \eqref{princprobu} (see \cite{CaMiNo} again or Lemma 3.5 in \cite{antolinz}). By the uniqueness of minimizers of \eqref{princprobu}, $w=u$ and $E_\gl= \{u<\gl\}$ for almost every $\gl$ (and then for every $\gl$ by an approximation procedure).

\noindent By Proposition \ref{contconv}, the function $\lambda \rightarrow \vgamma{E_\gl}$ is continuous on $]\overline{\lambda},+\infty[$ and nondecreasing. Together with the inclusion property of the $E_\gl$ this implies the uniqueness of the minimizers of $(P_\gl)$. Moreover, the sets $E_\lambda$ solve the problem:
\[\min_{\vgamma{E_\gl}=\vgamma{E}} \Pgamma{E}+ \int_E g\, d\gga.\]
Vice-versa, if $E_v$ solves \eqref{princprob} and $v>\overline{v}$ then there exists $\gl>\overline{\gl}$ such that $\vgamma{E_\gl}=v$ and as $E_v$ solves $(P_\gl)$ we get $E_v=E_\gl$. 
\end{proof}
\begin{remarque}\label{remperi}\rm
 If $F:H\to \R$ is homogeneous of degree one and such that 
\[c \normH{h} \le F(h) \le C \normH{h} \qquad \forall h \in H,\]
then $F$ satisfies (H2) and we can define the anisotropic perimeter $P_F$ by 
\[P_F(E):= \int_X F(D_\gga \chi_E).\]
Repeating verbatim the proof of \cite[Section 5.5]{EG}, (and using that smooth cylindrical functions are dense in $BV_\gga(X)$ by Proposition \ref{relax}), we still have a coarea formula,
\[\int_X F(D_\gga u) =\int_{\R} P_F( \{u <t\}) \ dt \qquad \forall u \in BV_\gga(X).\]
Using Theorem \ref{convlinear}, it is then not difficult to extend Theorem \ref{CaMiNointro} and Theorem \ref{main} to these anisotropic perimeters $P_F$. \\

\noindent Notice that in the Wiener space, the solution of the Wulff problem
\begin{equation}\label{wulffW}
 \min_{\gga(E)=v} P_F(E)
\end{equation}
is quite simple. If $F$ attains its minimum on the sphere at some direction $\nu_{\min}$ then by the isoperimetric inequality, if $E_{\nu_{\min}}$ is the half-space of volume $v$ and normal $\nu_{\min}$ and $E$ is any other set with volume $v$,
\[P_F(E_{\nu_{\min}})= F(\nu_{\min}) P_\gga(E_{\nu_{\min}})\le F(\nu_{\min}) P_\gga(E)\le P_F(E)\]
and thus $E_{\nu_{\min}}$ is the minimizer of \eqref{wulffW}. If instead $F$ does not attain its minimum on the sphere, there is no solution to \eqref{wulffW}.
\end{remarque}

\noindent We can finally state a simple corollary.
\begin{cor}
Let $g$ be a convex function in $\Ldeu$ and let 
\[F(E)=\Pgamma{E}+\int_E g \,d\gga. \]
Then two situations can occur:
\begin{itemize}
\item If $\min F<0$ then there exists a unique non-empty minimizer of $F$. Moreover this minimizer is convex.
\item If $\min F=0$ then there exists at most one non-empty minimizer of $F$ which is then convex.
\end{itemize}
\end{cor}

\begin{proof}
The two possibilities corresponds respectively to $\overline{\gl}<0$ and $\overline{\gl}\ge 0$.
\end{proof}

\noindent{\bf Acknowledgements.} 
The second author greatfully acknowledge the hospitality of the Scuola Normale Superiore di Pisa where this work has been initiated.
The third author wish to thank Michele Miranda for an interesting discussion regarding Remark \ref{remperi}.

\end{document}